\documentclass[11pt,reqno]{amsart}
\usepackage{amssymb,latexsym}
\usepackage{fancyhdr,amssymb}
\usepackage{hyperref}
\usepackage{listings}
\usepackage{amsmath}
\usepackage[T1]{fontenc}
\theoremstyle{theorem}
\usepackage{comment}
\newtheorem{theorem}{Theorem}[section]
\newtheorem{lemma}[theorem]{Lemma}
\newtheorem{proposition}[theorem]{Proposition}

\theoremstyle{definition}
\newtheorem{definition}[theorem]{Definition}

\theoremstyle{remark}
\newtheorem{remark}{Remark}

\theoremstyle{example}

\title{Cyclic derivations, species realizations and potentials}
\author{Daniel L\'{o}pez-Aguayo}
\address{Tecnologico de Monterrey, Escuela de Ingenier\'{i}a y Ciencias, Hidalgo, M\'{e}xico.}
\email{dlopez.aguayo@itesm.mx}

\begin{document}
\maketitle
\begin{abstract} 
In this survey paper we give an overview of a generalization, introduced by R. Bautista and the author, of the theory of mutation of quivers with potential developed in 2007 by Derksen-Weyman-Zelevinsky. This new construction allows us to consider finite dimensional semisimple $F$-algebras, where $F$ is any field. We give a brief account of the results concerning this generalization and its main consequences.
\end{abstract}
\section{Introduction} 
Since the development of the theory of quivers with potentials created by Derksen-Weyman-Zelevinsky in \cite{5}, the search for a general concept of \emph{mutation of a quiver with potential} has drawn a lot of attention. The theory of quivers with potentials has proven useful in many subjects of mathematics such as cluster algebras, Teichm\"{u}ller theory, KP solitons, mirror symmetry, Poisson geometry, among many others. There have been different generalizations of the notion of a quiver with potential and mutation where the underlying $F$-algebra, $F$ being a field, is replaced by more general algebras, see \cite{4,8,10}. 
This paper is organized as follows. In Section \ref{section2}, we review the preliminaries taken from \cite{1} and \cite{9}. Instead of working with an usual quiver, we consider the completion of the tensor algebra of $M$ over $S$, where $M$ is an $S$-bimodule and $S$ is a finite dimensional semisimple $F$-algebra. We will then see how to construct a cyclic derivation, in the sense of Rota-Sagan-Stein \cite{12}, on the completion of the tensor algebra of $M$. Then we introduce a natural generalization of the concepts of potential, right-equivalence and cyclical equivalence as defined in \cite{5}. 
In Section \ref{section3}, we describe a generalization of the so-called \emph{Splitting theorem} (\cite[Theorem 4.6]{5}) and see how this theorem allows us to lift the notion of mutation of a quiver with potential to this more general setting. 
Finally, in Section \ref{section4}, we recall the notion of species realizations and describe how the generalization given in \cite{1} allows us to give a partial affirmative answer to a question raised by J. Geuenich and D. Labardini-Fragoso in \cite{7}.

\pagebreak
\section{Preliminaries} \label{section2}
The following material is taken from \cite{1} and \cite{9}.

\begin{definition} Let $F$ be a field and let $D_{1},\ldots,D_{n}$ be division rings, each containing $F$ in its center and of finite dimension over $F$. Let $S=\displaystyle \prod_{i=1}^{n} D_{i}$ and let $M$ be an $S$-bimodule of finite dimension over $F$. Define the algebra of formal power series over $M$ as the set:
\begin{center}
$\mathcal{F}_{S}(M)=\left\{\displaystyle \sum_{i=0}^{\infty} a(i): a(i) \in M^{\otimes i}\right\}$
\end{center}
where $M^{0}=S$. Note that $\mathcal{F}_{S}(M)$ is an associative unital $F$-algebra where the product is the one obtained by extending the product of the tensor algebra $T_{S}(M)=\displaystyle \bigoplus_{i=0}^{\infty} M^{\otimes i}$.
\end{definition}

Let $\{e_{1},\ldots,e_{n}\}$ be a complete set of primitive orthogonal idempotents of $S$.

\begin{definition} An element $m \in M$ is legible if $m=e_{i}me_{j}$ for some idempotents $e_{i},e_{j}$ of $S$.
\end{definition}

\begin{definition} Let $Z=\displaystyle \sum_{i=1}^{n} Fe_{i} \subseteq S$. We say that $M$ is $Z$-freely generated by a $Z$-subbimodule $M_{0}$ of $M$ if the map $\mu_{M}: S \otimes_{Z} M_{0} \otimes_{Z} S \rightarrow M$ given by $\mu_{M}(s_{1} \otimes m \otimes s_{2})=s_{1}ms_{2}$ is an isomorphism of $S$-bimodules. In this case we say that $M$ is an $S$-bimodule which is $Z$-freely generated.
\end{definition}

Throughout this paper we will assume that $M$ is $Z$-freely generated by $M_{0}$.

\begin{definition} Let $A$ be an associative unital $F$-algebra. A cyclic derivation, in the sense of Rota-Sagan-Stein \cite{12}, is an $F$-linear function $\mathfrak{h}: A \rightarrow \operatorname{End}_{F}(A)$ such that:
\begin{equation}
\mathfrak{h}(f_{1}f_{2})(f)=\mathfrak{h}(f_{1})(f_{2}f)+\mathfrak{h}(f_{2})(ff_{1})
\end{equation}
for all $f,f_{1},f_{2} \in A$. Given a cyclic derivation $\mathfrak{h}$, we define the associated cyclic derivative $\delta: A \rightarrow A$ as $\delta(f)=\mathfrak{h}(f)(1)$.
\end{definition}

We now construct a cyclic derivative on $\mathcal{F}_{S}(M)$. First, we define a cyclic derivation on the tensor algebra $A=T_{S}(M)$ as follows. Consider the map:
\begin{center}
$\hat{u}: A \times A \rightarrow A$
\end{center}
given by $\hat{u}(f,g)=\displaystyle \sum_{i=1}^{n} e_{i}gfe_{i}$ for every $f,g \in A$. This is an $F$-bilinear map which is $Z$-balanced. By the universal property of the tensor product, there exists a linear map $u: A \otimes_{Z} A \rightarrow A$ such that $u(a \otimes b)=\hat{u}(a,b)$. 
Now we define an $F$-derivation $\Delta: A \rightarrow A \otimes_{Z} A$ as follows. For $s \in S$, we define $\Delta(s)=1 \otimes s - s \otimes 1$; for $m \in M_{0}$, we set $\Delta(m)=1 \otimes m$. Then we define $\Delta: M \rightarrow T_{S}(M)$ by:
\begin{center}
$\Delta(s_{1}ms_{2})=\Delta(s_{1})ms_{2}+s_{1}\Delta(m)s_{2}+s_{1}m\Delta(s_{2})$
\end{center}
for $s_{1},s_{2} \in S$ and $m \in M_{0}$.
Note that the above map is well-defined since $M \cong S \otimes_{Z} M_{0} \otimes_{Z} S$ via the multiplication map $\mu_{M}$. Once we have defined $\Delta$ on $M$, we can extend it to an $F$-derivation on $A$. Now we define $\mathfrak{h}: A \rightarrow \operatorname{End}_{F}(A)$ as follows:
\begin{center}
$\mathfrak{h}(f)(g)=u(\Delta(f)g)$
\end{center}
We have
\begin{align*}
\mathfrak{h}(f_{1}f_{2})(f)&=u(\Delta(f_{1}f_{2})f) \\
&=u(\Delta(f_{1})f_{2}f)+u(f_{1}\Delta(f_{2})f) \\
&=u(\Delta(f_{1})f_{2}f)+u(\Delta(f_{2})ff_{1}) \\
&=\mathfrak{h}(f_{1})(f_{2}f)+\mathfrak{h}(f_{2})(ff_{1}).
\end{align*}
It follows that $\mathfrak{h}$ is a cyclic derivation on $T_{S}(M)$. We now extend $\mathfrak{h}$ to $\mathcal{F}_{S}(M)$. Let $f,g \in \mathcal{F}_{S}(M)$, then $\mathfrak{h}(f(i))(g(j)) \in M^{\otimes(i+j)}$; thus we define $\mathfrak{h}(f)(g)(l)=\displaystyle \sum_{i+j=l} \mathfrak{h}(f(i))(g(j))$ for every $non$-negative integer $l$.

In \cite[Proposition 2.6]{9}, it is shown that the $F$-linear map $\mathfrak{h}: \mathcal{F}_{S}(M) \rightarrow \operatorname{End}_{F}(\mathcal{F}_{S}(M))$ is a cyclic derivation. Using this fact we obtain a cyclic derivative $\delta$ on $\mathcal{F}_{S}(M)$ given by
\begin{center}
$\delta(f)=\mathfrak{h}(f)(1)$.
\end{center}

\begin{definition} Let $\mathcal{C}$ be a subset of $M$. We say that $\mathcal{C}$ is a right $S$-local basis of $M$ if every element of $C$ is legible and if for each pair of idempotents $e_{i},e_{j}$ of $S$, we have that $C \cap e_{i}Me_{j}$ is a $D_{j}$-basis for $e_{i}Me_{j}$.
\end{definition}
We note that a right $S$-local basis $\mathcal{C}$ induces a dual basis $\{u,u^{\ast}\}_{u \in \mathcal{C}}$, where $u^{\ast}: M_{S} \rightarrow S_{S}$ is the morphism of right $S$-modules defined by $u^{\ast}(v)=0$ if $v \in \mathcal{C} \setminus \{u\}$; and $u^{\ast}(u)=e_{j}$ if $u=e_{i}ue_{j}$.

Let $T$ be a $Z$-local basis of $M_{0}$ and let $L$ be a $Z$-local basis of $S$. The former means that for each pair of distinct idempotents $e_{i}$,$e_{j}$ of $S$, $T \cap e_{i}Me_{j}$ is an $F$-basis of $e_{i}M_{0}e_{j}$; the latter means that $L(i)=L \cap e_{i}S$ is an $F$-basis of the division algebra $e_{i}S=D_{i}$. It follows that the non-zero elements of the set $\{sa: s \in L, a \in T\}$ form a right $S$-local basis of $M$. Therefore, for every $s \in L$ and $a \in T$, we have the map $(sa)^{\ast} \in \operatorname{Hom}_{S}(M_{S},S_{S})$ induced by the dual basis.

\begin{definition} Let $\mathcal{D}$ be a subset of $M$. We say that $\mathcal{D}$ is a left $S$-local basis of $M$ if every element of $\mathcal{D}$ is legible and if for each pair of idempotents $e_{i},e_{j}$ of $S$, we have that $\mathcal{D} \cap e_{i}Me_{j}$ is a $D_{i}$-basis for $e_{i}Me_{j}$.
\end{definition}

Let $\psi$ be any element of $\mathrm{Hom}_{S}(M_{S},S_{S})$. We will extend $\psi$ to an $F$-linear endomorphism of $\mathcal{F}_{S}(M)$, which we will denote by $\psi_{\ast}$.

First, we define $\psi_{\ast}(s)=0$ for $s\in S$; and for $M^{\otimes l}$, where $l \geq 1$, we define $\psi_{\ast}(m_{1}\otimes \dotsm \otimes m_{l})=\psi (m_{1})m_{2} \otimes \cdots \otimes m_{l}\in M^{\otimes (l-1)}$ for $m_{1},\dots,m_{l}\in M$. Finally, for $f\in \mathcal{F}_{S}(M)$ we define $\psi_{\ast} (f)\in \mathcal{F}_{S}(M)$ by setting
$\psi_{\ast} (f)(l-1)=\psi_{\ast}(f(l))$ for each integer $l>1$. Then we define
\begin{center}
$\psi_{\ast} (f)=\displaystyle \sum _{l=0}^{\infty }\psi_{\ast} (f(l)).$
\end{center}

\begin{definition} Let $\psi \in M^{*}=\mathrm{Hom}_{S}(M_{S},S_{S})$ and $f\in \mathcal{F}_{S}(M)$. We define $\delta _{\psi }:\mathcal{F}_{S}(M)\rightarrow \mathcal{F}_{S}(M)$ as
\begin{center}
$\delta _{\psi }(f)=\psi_{\ast}(\delta (f))=\displaystyle \sum _{l=0}^{\infty }\psi_{\ast}(\delta (f(l))).$
\end{center}
\end{definition}

\begin{definition} Given an $S$-bimodule $N$ we define the \emph{cyclic part} of $N$ as $N_{cyc}:=\displaystyle \sum_{j=1}^{n} e_{j}Ne_{j}$.
\end{definition}
\begin{definition} A \emph{potential}  $P$ is an element of $\mathcal{F}_{S}(M)_{cyc}$. 
\end{definition}

Motivated by the \emph{Jacobian ideal} introduced in \cite{5}, we define an analogous two-sided ideal of $\mathcal{F}_{S}(M)$. 

For each legible element $a$ of $e_{i}Me_{j}$, we let $\sigma(a)=i$ and $\tau(a)=j$.

\begin{definition} Let $P$ be a potential in $\mathcal{F}_{S}(M)$, we define a two-sided ideal $R(P)$ as the closure of the two-sided ideal of $\mathcal{F}_{S}(M)$ generated by all the elements $X_{a^{\ast}}(P)=\displaystyle \sum_{s \in L(\sigma(a))} \delta_{(sa)^{\ast}}(P)s$, $a \in T$.
\end{definition}

In \cite[Theorem 5.3]{1}, it is shown that $R(P)$ is invariant under algebra isomorphisms that fix pointwise $S$. Furthermore, one can show that $R(P)$ is independent of the choice of the $Z$-subbimodule $M_{0}$ and also independent of the choice of $Z$-local bases for $S$ and $M_{0}$.

\begin{definition} An algebra with potential is a pair $(\mathcal{F}_{S}(M),P)$ where $P$ is a potential in $\mathcal{F}_{S}(M)$ and $M_{cyc}=0$.
\end{definition}

We denote by $[\mathcal{F}_{S}(M),\mathcal{F}_{S}(M)]$ the closure in $\mathcal{F}_{S}(M)$ of the $F$-subspace generated by all the elements of the form $[f,g]=fg-gf$ with $f,g\in \mathcal{F}_{S}(M).$

\begin{definition} Two potentials $P$ and $P'$ are called cyclically equivalent if $P-P' \in  [\mathcal{F}_{S}(M),\mathcal{F}_{S}(M)]$.
\end{definition}

\begin{definition} We say that two algebras with potential $(\mathcal{F}_{S}(M),P)$ and $(\mathcal{F}_{S}(M'),Q)$ are right-equivalent if there exists an algebra isomorphism $\varphi: \mathcal{F}_{S}(M) \rightarrow \mathcal{F}_{S}(M')$, with $\varphi|_{S}=id_{S}$, such that $\varphi(P)$ is cyclically equivalent to $Q$. 
\end{definition}

The following construction follows the one given in \cite[p.20]{5}. Let $k$ be an integer in $[1,n]$ and $\bar{e}_{k}=1-e_{k}$. Using the $S$-bimodule $M$, we define a new $S$-bimodule $\mu_{k}M=\widetilde{M}$ as:
\begin{center}
$\widetilde{M}:=\bar{e}_{k}M\bar{e}_{k} \oplus Me_{k}M \oplus (e_{k}M)^{\ast} \oplus ^{\ast}(Me_{k})$
\end{center}
where $(e_{k}M)^{\ast}=\operatorname{Hom}_{S}((e_{k}M)_{S},S_{S})$, and $^{\ast}(Me_{k})=\operatorname{Hom}_{S}(_{S}(Me_{k}),_{S}S)$. One can show (see \cite[Lemma 8.7]{1}) that $\mu_{k}M$ is $Z$-freely generated. 

\begin{definition} Let $P$ be a potential in $\mathcal{F}_{S}(M)$ such that $e_{k}Pe_{k}=0$. Following \cite{5}, we define
\begin{center}
$\mu_{k}P:=[P]+\displaystyle \sum_{sa \in _{k}\hat{T},bt \in \tilde{T}_{k}}[btsa]((sa)^{\ast})(^{\ast}(bt))$
\end{center}
\end{definition}
where: 
\begin{align*}
_{k}\hat{T}&=\{sa: s \in L(k),a \in T \cap e_{k}M\} \\
\tilde{T}_{k}&=\{bt: b \in T \cap Me_{k}, t \in L(k)\}.  
\end{align*}

\section{Mutations and potentials} \label{section3}
Let $P=\displaystyle \sum_{i=1}^{N} a_{i}b_{i}+P'$ be a potential in $\mathcal{F}_{S}(M)$ where $A=\{a_{1},b_{1},\ldots,a_{N},b_{N}\}$ is contained in a $Z$-local basis $T$ of $M_{0}$ and $P' \in \mathcal{F}_{S}(M)^{\geq 3}$. Let $L_{1}$ denote the complement of $A$ in $T$, $N_{1}$ be the $F$-vector subspace of $M$ generated by $A$ and $N_{2}$ be the $F$-vector subspace of $M$ generated by $L_{1}$; then $M=M_{1} \oplus M_{2}$ as $S$-bimodules where $M_{1}=SN_{1}S$ and $M_{2}=SN_{S}S$. 

One of the main results proved in \cite{5} is the so-called \emph{Splitting theorem} (Theorem 4.6). Inspired by this result, the following theorem is proved in \cite{1}.

\begin{theorem} (\cite[Theorem 7.15]{1}) \label{split} There exists an algebra automorphism $\varphi: \mathcal{F}_{S}(M) \rightarrow \mathcal{F}_{S}(M)$ such that $\varphi(P)$ is cyclically equivalent to a potential of the form $\displaystyle \sum_{i=1}^{N} a_{i}b_{i}+P''$ where $P''$ is a reduced potential contained in the closure of the algebra generated by $M_{2}$ and $\displaystyle \sum_{i=1}^{N} a_{i}b_{i}$ is a trivial potential in $\mathcal{F}_{S}(M_{1})$.
\end{theorem} 

\begin{definition} Let $P\in \mathcal{F}(M)$ be a potential and $k$ an integer in $\{1,\ldots,n\}$. Suppose that there are no two-cycles passing through $k$. Using Theorem \ref{split}, one can see that $\mu _{k}P$ is right-equivalent to the direct sum of a trivial potential $W$ and a potential $Q$ in $\mathcal{F}_{S}(M)^{\geq 3}$. Following \cite{5}, we define the mutation of $P$ in the direction $k$, as $\overline{\mu }_{k}(P)=Q$.
\end{definition}

One of the main results of \cite{5} is that mutation at an arbitrary vertex is a well-defined involution on the set of right-equivalence classes of reduced quivers with potentials. In \cite{1}, the following analogous result is proved.

\begin{theorem} (\cite[Theorem 8.21]{1}) Let $P$ be a reduced potential such that the mutation $\overline{\mu}_{k}P$ is defined. Then $\overline{\mu}_{k}\overline{\mu}_{k}P$ is defined and it is right-equivalent to $P$.
\end{theorem}

\begin{definition} Let $k_{1}, \hdots, k_{l}$ be a finite sequence of elements of $\{1,\hdots,n\}$ such that $k_{p} \neq k_{p+1}$ for $p=1,\hdots, l-1$. We say that an algebra with potential $(\mathcal{F}_{S}(M),P)$ is $(k_{l},\hdots,k_{1})$-nondegenerate if all the iterated mutations $\bar{\mu}_{k_{1}}P$, $\bar{\mu}_{k_{2}}\bar{\mu}_{k_{1}}P, \hdots, \bar{\mu}_{k_{l}} \cdots \bar{\mu}_{k_{1}}P$ are $2$-acyclic. We say that $(\mathcal{F}_{S}(M),P)$ is nondegenerate if it is $(k_{l},\hdots, k_{1})$-nondegenerate for every sequence of integers as above.
\end{definition}
In \cite{5}, it is shown that if the underlying base field $F$ is uncountable then a nondegenerate quiver with potential exists for every underlying quiver. Motivated by this result, the following theorem is proved in \cite{9}.

\begin{theorem} (\cite[Theorem 3.5]{9}) \label{teouncount} Suppose that the underlying field $F$ is uncountable, then $\mathcal{F}_{S}(M)$ admits a nondegenerate potential.
\end{theorem}

\section{Species realizations} \label{section4}
We begin this Section by recalling the definition of \emph{species realization} of a skew-symmetrizable integer matrix, in the sense of \cite{7} (Definition 2.22).

\begin{definition} \label{especies} Let $B=(b_{ij}) \in \mathbb{Z}^{n \times n}$ be a skew-symmetrizable matrix, and let $I=\{1,\hdots, n\}$. A species realization of $B$ is a pair $(\mathbf{S},\mathbf{M})$ such that:
\begin{enumerate}
\item $\mathbf{S}=(F_{i})_{i \in I}$ is a tuple of division rings;
\item $\mathbf{M}$ is a tuple consisting of an $F_{i}-F_{j}$ bimodule $M_{ij}$ for each pair $(i,j) \in I^{2}$ such that $b_{ij}>0$;
\item for every pair $(i,j) \in I^{2}$ such that $b_{ij}>0$, there are $F_{j}-F_{i}$-bimodule isomorphisms
\begin{center}
$\operatorname{Hom}_{F_{i}}(M_{ij},F_{i}) \cong \operatorname{Hom}_{F_{j}}(M_{ij},F_{j})$;
\end{center}
\item for every pair $(i,j) \in I^{2}$ such that $b_{ij}>0$ we have $\operatorname{dim}_{F_{i}}(M_{ij})=b_{ij}$ and $\operatorname{dim}_{F_{j}}(M_{ij})=-b_{ji}$.
\end{enumerate}
\end{definition}

In \cite[p.29]{1}, we impose the following condition on each of the bases $L(i)$. For each $s,t \in L(i)$:
\begin{equation} \label{eq2}
e_{i}^{\ast}(st^{-1}) \neq 0 \ \text{implies} \ s=t \ \text{and} \ e_{i}^{\ast}(s^{-1}t) \neq 0 \ \text{implies} \ s=t
\end{equation}
where $e_{i}^{\ast}: D_{i} \rightarrow F$ denotes the standard dual map corresponding to the basis element $e_{i} \in L(i)$.

In \cite[p.14]{7}, motivated by the seminal paper \cite{5}, J. Geuenich and D. Labardini-Fragoso raise the following question: \\

\textbf{Question} \cite[Question 2.23]{7} Can a mutation theory of species with potential be defined so that every skew-symmetrizable matrix $B$ have a species realization which admit a nondegenerate potential? \\

In \cite[Corollary 3.6]{9} a partially affirmative answer to Question $2.23$ is given by proving the following: let $B=(b_{ij}) \in \mathbb{Z}^{n \times n}$ be a skew-symmetrizable matrix with skew-symmetrizer $D=\operatorname{diag}(d_{1},\ldots,d_{n})$. If $d_{j}$ divides $b_{ij}$ for every $j$ and every $i$, then the matrix $B$ can be realized by a species that admits a nondegenerate potential.

We now give an example (\cite[p.8]{9}) of a class of skew-symmetrizable $4 \times 4$ integer matrices, which are not globally unfoldable nor strongly primitive, and that have a species realization admitting a nondegenerate potential. This gives an example of a class of skew-symmetrizable $4 \times 4$ integer matrices which are not covered by \cite{8}. 

Let 
\begin{equation} \label{eq3}
B
=\begin{bmatrix}
0 & -a & 0 & b \\
1 & 0 & -1 & 0 \\
0 & a & 0 & -b \\
-1 & 0 & 1 & 0
\end{bmatrix}
\end{equation}
where $a,b$ are positive integers such that $a<b$, $a$ does not divide $b$ and $\gcd(a,b) \neq 1$. \\

Note that there are infinitely many such pairs $(a,b)$. For example, let $p$ and $q$ be primes such that $p<q$. For any $n \geq 2$, define $a=p^{n}$ and $b=p^{n-1}q$. Then $a<b$, $a$ does not divide $b$ and $\gcd(a,b)=p^{n-1} \neq 1$. Note that $B$ is skew-symmetrizable since it admits $D=\operatorname{diag}(1,a,1,b)$ as a skew-symmetrizer. \\

\begin{remark} By \cite[Example 14.4]{8} we know that the class of all matrices given by \eqref{eq3} does \emph{not} admit a global unfolding. Moreover, since we are not assuming that $a$ and $b$ are coprime, then such matrices are not strongly primitive; hence they are not covered by \cite{8}. 
\end{remark}

We have the following
\begin{proposition} (\cite[Proposition 5.2]{9})
The class of all matrices given by \eqref{eq3} are not globally unfoldable nor strongly primitive, yet they can be realized by a species admitting a nondegenerate potential.  
\end{proposition}

By Theorem \ref{teouncount}, we know that a nondegenerate potential exists provided the underlying field $F$ is uncountable. If $F$ is infinite (but not necessarily uncountable) one can show that $\mathcal{F}_{S}(M)$ admits ``locally'' nondegenerate potentials. More precisely, we have

\begin{proposition} (\cite[Proposition 12.5]{1}) \label{infinite} Let $F$ be an infinite field and let $k_{1},\ldots,k_{l}$ be an arbitrary sequence of elements of $\{1,\ldots,n\}$. Then there exists a potential $P \in \mathcal{F}_{S}(M)$ such that the mutation $\overline{\mu}_{k_{l}} \cdots \overline{\mu}_{k_{1}}P$ exists.
\end{proposition}

We conclude the paper by giving an example of a class of skew-symmetrizable $4 \times 4$ integer matrices that have a species realization via field extensions of the rational numbers. Although in this case we cannot guarantee the existence of a nondegenerate potential, we can guarantee (by Proposition \ref{infinite}) the existence of ``locally'' nondegenerate potentials.

First we require some definitions.

\begin{definition} Let $E/F$ be a finite field extension. An $F$-basis of $E$, as a vector space, is said to be semi-multiplicative if the product of any two elements of the basis is an $F$-multiple of another basis element.
\end{definition}

It can be shown that every extension $E/F$ which has a semi-multiplicative basis satisfies (\ref{eq2}).

\begin{definition} A field extension $E/F$ is called a simple radical extension if $E=F(a)$ for some $a \in E$, with $a^n \in F$ and $n \geq 2$.
\end{definition}

Note that if $E/F$ is a simple radical extension then $E$ has a semi-multiplicative $F$-basis.

\begin{definition} A field extension $E/F$ is a radical extension if there exists a tower of fields $F=F_{0} \subseteq F_{1} \ldots \subseteq F_{l}=E$ such that $F_{i}/F_{i-1}$ is a simple radical extension for $i=1,\ldots, l$. 
\end{definition}

As before, let
\begin{equation} 
B
=\begin{bmatrix}
0 & -a & 0 & b \\
1 & 0 & -1 & 0 \\
0 & a & 0 & -b \\
-1 & 0 & 1 & 0
\end{bmatrix}
\end{equation}
but without imposing additional conditions on $a$ or $b$.

\begin{proposition} \label{rationals} Let $n,m \geq 2$. The matrix $B$ admits a species realization $(\mathbf{S},\mathbf{M})$ where $\mathbf{M}$ is a $Z$-freely generated $S$-bimodule and $S$ satisfies (\ref{eq2}).
\end{proposition}

\begin{proof} To prove this we will require the following result (cf. \cite[Theorem 14.3.2]{11}).

\begin{lemma} \label{lemroots} Let $n \geq 2$, $p_{1},\ldots,p_{m}$ be distinct primes and let $\mathbf{Q}$ denote the set of all rational numbers. Let $\zeta_{n}$ be a primitive $nth$-root of unity. Then
\begin{center}
$[\mathbf{Q}(\zeta_{n})(\sqrt[n]{p_{1}},\ldots,\sqrt[n]{p_{m}}): \mathbf{Q}(\zeta_{n})]=n^{m}$
\end{center}
\end{lemma}
Now we continue with the proof of Proposition \ref{rationals}. Let $F=\mathbf{Q}(\zeta_{n})$ be the base field and let $p_{1}$ be an arbitrary prime. By Lemma \ref{lemroots}, $F_{2}=F(\sqrt[n]{p_{1}})/F$ has degree $n$. Now choose $m-1$ distinct primes $p_{2},p_{3},\ldots,p_{m}$ and also distinct from $p_{1}$. Define $F_{4}=F(\sqrt[n]{p_{1}},\sqrt[n]{p_{2}},\ldots,\sqrt[n]{p_{m}})/F$, then by Lemma \ref{lemroots}, $F_{4}$ has degree $n^{m}$.
Let $S=F \oplus F_{2} \oplus F \oplus F_{4}$ and $Z=F \oplus F \oplus F \oplus F$. Since $F/\mathbf{Q}$ is a simple radical extension then it has a semi-multiplicative basis; thus it satisfies (\ref{eq2}). On the other hand, note that $F_{2}/\mathbf{Q}(\zeta_{n})$ and $F_{4}/\mathbf{Q}(\zeta_{n})$ are radical extensions. Using \cite[Remark 6, p.29]{1} we get that both $F_{2}$ and $F_{4}$ satisfy (\ref{eq2}); hence, it is always possible to choose a $Z$-local basis of $S$ satisfying (\ref{eq2}). Finally, for each $b_{ij}>0$, define $e_{i}Me_{j}=(F_{i} \otimes_{F} F_{j})^{\frac{b_{ij}}{d_{j}}}=F_{i} \otimes_{F} F_{j}$. It follows that $(\mathbf{S},\mathbf{M})$ is a species realization of $B$. 
\end{proof}


\end{document}